\documentclass{article}
\usepackage{amsmath,amsthm,amscd,amssymb}
\usepackage[mathscr]{eucal}
\usepackage{amssymb}
\usepackage{latexsym}

\theoremstyle{definition}
\newtheorem{Def}{Definition}[section]
\newtheorem{Thm}[Def]{Theorem}

\newtheorem{Rem}[Def]{Remark}
\newtheorem{Ex}[Def]{Example}

\newtheorem{Prob}[Def]{Problem}
\newtheorem{Lem}[Def]{Lemma}

\numberwithin{equation}{section}

\begin{document}

\title{On $p$-adic Siegel modular forms of non-real Nebentypus of degree $2$}

\author{Toshiyuki Kikuta}
\maketitle

%\date{}

%\dedicatory{}

%\keywords{}

\begin{abstract}
We show that all Siegel modular forms of non-real Nebentypus for $\Gamma_0^{(2)}(p)$ are $p$-adic Siegel modular forms by using a Maass lift. 
\end{abstract}

\maketitle

\section{Introduction}
\label{intro}
In \cite{Se}, Serre defined the notion of $p$-adic modular forms and applied it to the construction of a $p$-adic $L$-function. Recently, several people attempted to generalize this notion to that of the case of several variables. In particular, B\"ocherer-Nagaoka \cite{Bo-Na} defined the $p$-adic Siegel modular forms and showed that all Siegel modular forms with level $p$ and real Nebentypus are $p$-adic Siegel modular forms. The aim of this paper is to generalize it to the case of non-real Nebentypus. 

We state our results more precisely. Let $k$ be a positive integer, $p$ an odd prime and $\chi $ a Dirichlet character modulo $p$ with $\chi (-1)=(-1)^k$. For the congruence subgroup $\Gamma _0^{(n)}(p)$ of the symplectic group $\Gamma _n=Sp_n(\mathbb{Z})$, we denote by $M_{k}(\Gamma _0^{(n)}(p),\chi )$ the space of corresponding Siegel modular forms of weight $k$ and character $\chi $. For a subring $R$ of $\mathbb{C}$, let $M_{k}(\Gamma _0^{(n)}(p),\chi )_{R}\subset M_{k}(\Gamma _0^{(n)}(p),\chi )$ denote the $R$-module of all modular forms whose Fourier coefficients belong to $R$. Let $\mu _{p-1}$ denote the group of the ($p-1$)-th roots of unity in $\mathbb{C}^{\times}$. We fix an embedding $\sigma $ from $\mathbb{Q}(\mu_{p-1})$ to $\mathbb{Q}_p$ (see Subsection \ref{Cyc}). The following theorem is our main result:    
\begin{Thm} 
\label{ThmM}
For any modular form $F\in M_k(\Gamma _0^{(2)}(p),\chi )_{\mathbb{Q}(\mu _{p-1})}$, $F^{\sigma }$ is a $p$-adic Siegel modular form. In other words, there exists a sequence of full modular forms $\{G_{k_m}\}$ such that 
\begin{align*}
\lim _{m\to \infty }G_{k_m}=F^\sigma \quad (p{\text -adically}).
\end{align*}  
\end{Thm}
In Section \ref{Proof}, we prove Theorem \ref{ThmM}. The key point of the proof is the following existence theorem: Let $\omega $ be the Tichm\"uler character on $\mathbb{Z}_p$. 
\begin{Thm}
\label{Thm2}
We take $\alpha \in \mathbb{Z}/(p-1)\mathbb{Z}$ such that $\chi ^{\sigma }=\omega ^{\alpha }$. Then there exists a sequence of modular forms $\{G_{k_m}\in M_{k_m}(\Gamma _0^{(2)}(p),\chi )_{\mathbb{Q}(\mu _{p-1})}\}$ such that 
\begin{align*}
\lim _{m\to \infty } G_{k_m}^{\sigma }=1\quad (p{\text -adically}). 
\end{align*} 
\end{Thm} 
\begin{Rem}
\label{Rem}
If we denote by ${\boldsymbol X}:=\mathbb{Z}_p\times \mathbb{Z}/(p-1)\mathbb{Z}$ the group of the weights of $p$-adic Siegel modular forms, the sequence $\{k_m\}$ of weights in Theorem \ref{Thm2} converges automatically $(0,-\alpha )$ in ${\boldsymbol X}$ by the results \cite{Bo-Na2,Ich,Se}. 
\end{Rem}

%%%%%%%%%%%%%%%%%%%%%%%%%%%%%%%%%%%%%%%%%%%%%%%%%%%%%%%%%%
\section{Preliminaries}
\subsection{Siegel modular forms}
Let $\mathbb{H}_n$ be the Siegel upper-half space of degree $n$. The Siegel modular group $\Gamma _n=Sp_n(\mathbb{Z})$ acts on $\mathbb{H}_n$ by the generalized fractional transformation
\begin{align*}
MZ:=(AZ+B)(CZ+D)^{-1},\quad for\ M=\begin{pmatrix}A & B \\ C & D \end{pmatrix}\in \Gamma _n.
\end{align*}
Let $N$ be a positive integer. The congruence subgroup $\Gamma _0^{(n)}(N)$ is defined by 
\begin{align*}
\Gamma _0^{(n)}(N):=\left\{ \begin{pmatrix}A & B \\ C & D \end{pmatrix}\in \Gamma _n \big| C\equiv O_n \bmod{N} \right\}. 
\end{align*}
Let $\chi $ be a Dirichlet character modulo $N$. The space $M_k(\Gamma _0^{(n)}(N),\chi )$ of Siegel modular forms of weight $k$ and character $\chi $ consists of all of holomorphic functions $f:\mathbb{H}_n\rightarrow \mathbb{C}$ satisfying  
\begin{align*}
f(MZ)=\chi (\det D)\det (CZ+D)^kf(Z),\quad for\ M=\begin{pmatrix}A & B \\ C & D \end{pmatrix}\in \Gamma _0^{(n)}(N). 
\end{align*} 
If $\chi $ is trivial, we write as $M_k(\Gamma _0^{(n)}(N))=M_k(\Gamma _0^{(n)}(N),\chi )$ simply.
If $f\in M_k(\Gamma _0^{(n)}(N),\chi )$ then $f$ has a Fourier expansion of the form
\begin{align*}
f=\sum _{O\le T\in \Lambda _n}a_f(T)e^{2\pi i {\rm tr}(TZ)},
\end{align*}
where $T$ runs over all elements of semi-positive definite of $\Lambda _n$ and 
\begin{align*}
\Lambda _n:=\{T=(t_{ij})\in Sym_n(\mathbb{Q})|t_{ii}\in \mathbb{Z},\ 2t_{ij}\in \mathbb{Z} \}. 
\end{align*} 
In this paper, we mainly deal with the case where $N$ is a prime. 

%%%%%%%%%%%%%%%%%%%%%%%%%%%%%%%%%%%%%%%%%%%%%%%%%%%%%%%%%%%%%%%%
\subsection{$p$-adic Siegel modular forms}

Let $v_p$ be the normalized additive valuation on $\mathbb{Q}_p$ as $v_p(p)=1$. We consider a formal power series of the form $f=\sum _{O\le T\in \Lambda _n}a(T)e^{2\pi i {\rm tr}(TZ)}$ with $a(T)\in \mathbb{Q}_p$. For more accurate interpretation of $f$, see \cite{Bo-Na,Bo-Na2}.

\begin{Def}
A formal power series $f=\sum _{O\le T\in \Lambda _n}a(T)e^{2\pi i {\rm tr}(TZ)}$ with $a(T)\in \mathbb{Q}_p$ called a $p$-$adic$ $Siegel$ $modular$ $form$ if there exists a sequence of full modular forms $\{g_m\}\subset M_{k_m}(\Gamma _2)_{\mathbb{Q}}$ such that $\lim _{m\to \infty }g_m=f$ ($p$-adically), where the limit means that $\inf _{T\in \Lambda _n} (v_p(a_{g_m}(T)-a(T)))\to \infty $ as $m\to \infty$.  
\end{Def}
In \cite{Bo-Na}, B\"ocherer and Nagaoka showed that
\begin{Thm}[B\"ocherer-Nagaoka \cite{Bo-Na}]
\label{Bo-Na}
Let $p$ be an odd prime. If $f\in M_k(\Gamma _0^{(n)}(p))_{\mathbb{Q}}$ then $f$ is a $p$-adic Siegel modular form. 
\end{Thm}

%\begin{Rem}
%\end{Rem}

%%%%%%%%%%%%%%%%%%%%%%%%%%%%%%%%%%%%%%%%%%%%%%%%%%%%%%%%%%%%%%%%
\subsection{Jacobi forms and their liftings}
In this subsection, we recall the known facts related Jacobi forms and their liftings. Since we do not need the general level case, we only consider the prime level case.  
 
Let $p$ be an odd prime and $\chi $ a Dirichlet character modulo $p$ with $\chi (-1)=(-1)^k$. Let $\phi $ be a Jacobi form of weight $k$, index $1$ and character $\chi $ with respect to $\Gamma _0^{(1)}(p)$. Then $\phi $ has a Fourier expansion of the form 
\begin{align*}
\phi (\tau ,z)=\sum _{n=0}^{\infty } \sum _{\substack{ r \in \mathbb{Z} \\ 4n-r^2\ge 0 }}c(n,r)q^n\zeta ^r,\quad for\ (\tau ,z)\in \mathbb{H}_1\times \mathbb{C}, 
\end{align*}
where $q:=e^{2 \pi i \tau }$ and $\zeta :=e^{2\pi i z}$. The Maass lift ${\mathcal M}\phi \in M_{k}(\Gamma ^{(2)}_{0}(p),\chi )$ of $\phi $ is described by
\begin{align*}
{\mathcal M}\phi (Z)&=\left( \frac{1}{2}L(1-k,\chi )+\sum _{n=1}^{\infty } \sum _{d|n}\chi (d)d^{k-1}q^n \right)c(0,0)\\
&+\sum _{l=1}^{\infty }\sum _{4nl-r^2\ge 0} \sum_{\substack{ d|(n,r,l) \\ (d,p)=1}} \chi (d)d^{k-1}c\left(\frac{nl}{d^2}, \frac{r}{d}\right) q^n\zeta ^rq'^l,\quad for\ Z=\begin{pmatrix}\tau & z \\ z & w \end{pmatrix}\in \mathbb{H}_2, 
\end{align*}
where $q':=e^{2\pi i w}$. 
This lift was studied by Ibukiyama. For the precise definitions of Jacobi forms with level and their liftings, see \cite{Ibu,Miz}. 

%%%%%%%%%%%%%%%%%%%%%%%%%%%%%%%%%%%%%%%%%%%%%%%%%%%%%%%%%%
\subsection{Embeddings from $\mathbb{Q}(\mu _{p-1})$ to $\mathbb{Q}_p$}
\label{Cyc}
In this subsection, we mention that how to determine the embeddings from $\mathbb{Q}(\mu _{p-1})$ to $\mathbb{Q}_p$. 

Let $\mu _{p-1}$ denote the group of the ($p-1$)-th roots of unity in $\mathbb{C}^{\times}$. Let us take a generator $\zeta _{p-1}$ of $\mu _{p-1}$ and consider the prime ideal factorization of $p$ in the ring $\mathbb{Z}[\zeta _{p-1}]$ of integers of $\mathbb{Q}(\mu _{p-1})$. Let $\Phi (X)\in \mathbb{Z}[X]$ be the minimal polynomial of $\zeta _{p-1}$, namely $\Phi (X)$ is the cyclotomic polynomial having the root $\zeta _{p-1}$. We can always decompose $\Phi (X)$ as the form $\Phi (X)\equiv q_1(X)\cdots q_r(X)$ mod $p$, where $r=\varphi (p-1)$, each $q_i(X)$ is a polynomial of degree one with $q_i(X) \not \equiv q_j (X)$ mod $p$. Then $p$ is decomposed as a product of $r$ prime ideals $\frak{p}_i:=(q_i(\zeta _{p-1}),p)$, namely we have the perfect decomposition 
\begin{align*}
(p)=\frak{p}_1\cdots \frak{p}_r=(q_1(\zeta _{p-1}), p)\cdots (q_r(\zeta _{p-1}),p). 
\end{align*} 
If we write $q_i(X)=X-d_i$ for some $d_i\in \mathbb{Z}$, then an embedding $\sigma _i$ from $\mathbb{Q}(\zeta _{p-1})$ to $\mathbb{Q}_p$ corresponding $\frak{p}_i$ is determined by $\sigma _i(\zeta _{p-1})=\omega (d_i)$. 
\begin{Ex}
(1) {\bf Case} $p=5$ ($\zeta_{4}=i$). \\
We see easily that $\Phi (X)=X^2+1\equiv (X-2)(X-3)$ mod $5$. Putting $\frak{p}_1:=(i-2,5)$ and $\frak{p}_2:=(i-3,5)$, then $(5)=\frak{p}_1\frak{p}_2$. In fact, $(i-2,5)=(i-2)$ and $(i-3,5)=(i+2)$. Hence, the embeddings $\sigma _i$ corresponding $\frak{p}_i$ are determined by $\sigma _1(i)=\omega (2)$ and $\sigma _2(i)=\omega (3)$. \\
\noindent(2) {\bf Case} $p=7$ ($\zeta_{6}=(1+\sqrt{3}i)/2$). \\
One has $\Phi (X)=X^2-X+1\equiv (X-3)(X-5)$ mod $7$. If we set $\frak{p}_1:=(\zeta _{6}-3,5)$ and $\frak{p}_2:=(\zeta _6-5,5)$, then $7=\frak{p}_1\frak{p}_2$. Hence, the embedding $\sigma _i$ are determined by $\sigma _1(\zeta _6)=\omega (3)$ and $\sigma _2(\zeta _6)=\omega (5)$. 
\end{Ex}

%%%%%%%%%%%%%%%%%%%%%%%%%%%%%%%%%%%%%%%%%%%%%%%%%%%%%%%%%%%%
\section{Proofs}
In this section, we prove our theorems. As introduced in Remark \ref{Rem}, let ${\boldsymbol X}:=\mathbb{Z}_p\times \mathbb{Z}/(p-1)\mathbb{Z}$ denote the group of $p$-adic Siegel modular forms. Following Serre's notation in \cite{Se}, let us write $\zeta ^{*}(s,u):=L_p(s,\omega ^{1-u})$ for $(s,u)\in {\boldsymbol X}$, where $L_p(s,\chi)$ is the Kubota-Leopoldt's $p$-adic $L$-function with character $\chi $ (e.g. \cite{Hida}).
\label{Proof}
\subsection{Proof of Theorem \ref{Thm2}}
We take a sequence $\{k_m=ap^m\}$ for $0<a\in \mathbb{Z}$ with $a\equiv -\alpha $ mod $p-1$. Note that $a$ is even or odd according as $\chi $ is even or odd. 

As in \cite{Ei-Za}, let $E^J_{k,1}(\tau ,z)$ be the normalized Siegel Jacobi Eisenstein series of weight $k$ and index $1$ (i.e. the constant term is $1$). It is known that its Fourier coefficients are in $\mathbb{Q}$. Moreover we denote by 
\begin{align}
\label{HeckeEisen}
&E^{(1)}_{k,\chi}=1+2L(1-k,\chi)^{-1}\sum _{n=1}^{\infty }\sum _{0<d|n}\chi (d)d^{k-1}q^n \in M_{k}(\Gamma _0^{(1)}(p),\chi ), \\
&E^{(1)}_{k}=1-\frac{2k}{B_k}\sum _{n=1}^{\infty }\sum _{0<d|n}d^{k-1}q^n \in M_{k}(\Gamma _1)
\end{align}
the normalized Eisenstein series of weight $k$ for $\Gamma _1$ and normalized Hecke's Eisenstein series of weight $k$ and character $\chi $ for $\Gamma _0^{(1)}(p)$, respectively. If we put 
\begin{align*}
\phi _{k_m}:=E_{a(p-2),\chi }^{(1)}E^{(1)}_{ap(p^{m-1}-1)}E^J_{2a,1}
\end{align*}
then we see that $\phi _{k_m}$ is a Jacobi form of weight $k_m$ and index $1$ with character $\chi $ for $\Gamma ^{(1)}_{0}(p)$. Here note that $E^{(1)}_{ap(p^{m-1}-1)}E^J_{2a,1}$ has rational Fourier coefficients. Moreover if we write its Fourier expansion as $\phi _{k_m}=\sum _{n,r}c_{k_m}(n,r)q^n\zeta ^r$, then $c_{k_m}(n,r)\in \mathbb{Q}(\mu _{p-1})$. Now we can prove
\begin{Lem}
\label{Lem2}
$\{\phi _{k_m}^{\sigma }\}$ converges in the formal power series ring $\mathbb{Q}_p[\![q,\zeta ]\!]$. Namely, each coefficient $c_{k_m}(n,r)^{\sigma }$ converges in $\mathbb{Q}_p$.  
\end{Lem}
\begin{proof}
Recall that 
\begin{align*}
\phi _{k_m}^{\sigma }&=(E_{a(p-2),\chi }^{(1)}E^{(1)}_{ap(p^{m-1}-1)}E^J_{2a,1})^{\sigma }=(E_{a(p-2),\chi }^{(1)})^{\sigma }E^{(1)}_{ap(p^{m-1}-1)}E^J_{2a,1}\in \mathbb{Q}_p[\![q,\zeta ]\!].  
\end{align*}
Hence we may only show that $\lim_{m\to \infty} E^{(1)}_{ap(p^{m-1}-1)} \in \mathbb{Q}_p[\![q]\!]$. To prove this, we consider the Eisenstein series 
\begin{align*}
G^{(1)}_{l_m}:=-\frac{B_{l_m}}{2l_m}E^{(1)}_{l_m}=-\frac{B_{l_m}}{2l_m}+\sum _{n=1}^{\infty }\sum _{0<d|n}d^{l_m-1}q^n,
\end{align*}
where we put $l_m:=ap(p^{m-1}-1)$ for the sake of simplicity. It is clear that $\{l_m\}$ is a Cauchy sequence. Hence there exists a limiting value $\lim _{m\to \infty }\sum _{0<d|n}d^{l_m-1}\in \mathbb{Q}_p$ for every $n\ge 1$. Since $l_m$ tends to $(-ap,0)\neq (0,0)$ in ${\boldsymbol X}$, we can apply Corollaire 2 in \cite{Se} to $G^{(1)}_{l_m}$. Therefore we see that the constant term also converges in $\mathbb{Q}_p$, namely
\begin{align*}
-\lim _{m\to \infty }\frac{B_{l_m}}{2l_m}\in \mathbb{Q}_p. 
\end{align*}
Now we shall show that this value is not zero. If $m\ge 2$ then $p-1|\l_m$. Hence the denominator of $B_{l_m}$ is divisible by $p$ according to Von-Staudt Clausen theorem. Moreover $p|\!|l_m$. Summarizing these facts, we see that the denominator of $B_{l_m}/2l_m$ is divisible by $p^2$ for every $m\ge 2$. It follows immediately from this property that  
\begin{align*}
-\lim _{m\to \infty }\frac{B_{l_m}}{2l_m}\neq 0. 
\end{align*}
Therefore we get  
\begin{align*}
\lim _{m\to \infty }E^{(1)}_{l_m}=\lim _{m\to \infty }\left(1-\frac{2l_m}{B_{l_m}}\sum _{n=1}^{\infty }\sum _{0<d|n}d^{l_m-1}q^n \right) \in \mathbb{Q}_p[\![q]\!].
\end{align*} 
This completes the proof of Lemma \ref{Lem2}. 
\end{proof} 
Let us return to the proof of Theorem \ref{Thm2}. Taking the Maass lift ${\mathcal M}\phi _{k_m}=:G_{k_m}\in M_{k_m}(\Gamma ^{(2)}_{0}(p),\chi )_{\mathbb{Q}(\mu _{p-1})}$, we have the following Fourier expansion  
\begin{align*}
G_{k_m}&=\frac{1}{2}L(1-k_m,\chi )+\sum _{n=1}^{\infty } \sum _{0<d|n}\chi (d)d^{k_m-1}(n)q^n\\
&+\sum _{l=1}^{\infty }\sum _{4nl-r^2\ge 0} \sum_{\substack{0<d|(n,r,l) \\ (p,d)=1}} \chi (d)d^{k_m-1}c_{k_m} \left(\frac{nl}{d^2}, \frac{r}{d}\right) q^n\zeta ^rq'^l. 
\end{align*} 
The $l>0$-th Fourier Jacobi coefficient is
\begin{align*}
\sum _{4nl-r^2\ge 0} \sum_{\substack{0<d|(n,r,l) \\ (p,d)=1}} \chi (d)d^{k_m-1}c_{k_m}\left(\frac{nl}{d^2}, \frac{r}{d}\right)q^n \zeta ^r.  
\end{align*}
Since $\chi (d)^{\sigma}=\omega (d)^{\alpha }=d^{\alpha }$, if we take $\sigma $, then 
\begin{align*}
\sum _{4nl-r^2\ge 0} \sum_{\substack{ 0<d|(n,r,l) \\ (p,d)=1}} d^{k_m+\alpha -1}c_{k_m}\left(\frac{nl}{d^2}, \frac{r}{d}\right)^{\sigma }q^n \zeta ^r. 
\end{align*} 
The first Fourier Jacobi coefficient is Hecke's Eisenstein series of weight $k_m$ and character $\chi $ in (\ref{HeckeEisen}). By a similar argument of Serre, we obtain 
\begin{align*}
\left( \frac{1}{2}L(1-{k_m},\chi )+\sum _{n=1}^{\infty } \sum _{0<d|n}\chi (d)d^{k_m-1}q^n \right)^\sigma = \zeta ^{*}(1-k_m,1-k_m-\alpha )+\sum _{n=1}^{\infty } \sum _{\substack{0<d|n \\ (p,d)=1}}d^{k_m+\alpha -1 }(n)q^n. 
\end{align*}

Finally, we set $G_{k_m}:=2L(1-k_m,\chi )^{-1}F_{k_m}$. Since $k_m$ tends to $(0,-\alpha )$ in ${\boldsymbol X}$, $(k_m,k_m+\alpha )$ tends to $(0,0)$ in ${\boldsymbol X}$. Note that $\zeta ^{*}(s,u)$ has a simple pole at $(1,1)$. Combining this fact with Lemma \ref{Lem2}, we see that $G_{k_m}^{\sigma }$ tends to $1$. In fact, the $q$-expansion of $G_{k_m}^{\sigma }$ is given by
\begin{align*}
G_{k_m}^{\sigma }&=1+\frac{1}{\zeta ^{*}(1-k_m,1-k_m-\alpha )}\sum _{n=1}^{\infty } \sum _{\substack{ 0<d|n \\ (p,d)=1}}d^{k_m+\alpha -1 }q^n \\
&+\frac{1}{\zeta ^{*}(1-k_m,1-k_m-\alpha )} \left( \sum _{l=1}^{\infty }\sum _{4nl-r^2\ge 0} \sum_{\substack{ 0<d|(n,r,l) \\ (p,d)=1}} d^{k_m+\alpha -1}c_{k_m}\left(\frac{nl}{d^2}, \frac{r}{d}\right)^{\sigma } q^n\zeta ^rq'^l\right). 
\end{align*}
This completes the proof of Theorem \ref{Thm2}. \qed

%%%%%%%%%%%%%%%%%%%%%%%%%%%%%%%%%%%%%%%%%%%%%%%%%%%%%%%%%%%%%%%%%%%%%%
\subsection{Proof of Theorem \ref{ThmM}}
In order to apply Serre's argument, we start with proving that 
\begin{Lem}
\label{Lem3}
Let $f\in M_{k}(\Gamma _0^{(n)}(p))_{\mathbb{Q}(\mu _{p-1})}$. Then $f$ is a $\mathbb{Q}(\mu _{p-1})$-linear combination of elements of $M_k(\Gamma _0^{(n)}(p))_{\mathbb{Q}}$.
\end{Lem}
\begin{proof}
It holds that $M_k(\Gamma _0^{(n)}(p))_\mathbb{C}=M_k(\Gamma _0^{(n)}(p))_\mathbb{Q}\otimes \mathbb{C}$ by Shimura's result \cite{Shi}. This fact tells us that $f\in M_k(\Gamma _0^{(n)}(p))_\mathbb{Q}(\mu _{p-1})$ is uniquely written in the form $f=\sum _{i=1}^{N}c_if_i$ for some $c_i\in \mathbb{C}$ and $f_i\in M_{k}(\Gamma _0^{(n)}(p))_{\mathbb{Q}}$. For each $\tau \in Aut(\mathbb{C}/\mathbb{Q}(\mu _{p-1}))$, $f^{\tau }=\sum _{i=1}^Nc_i^{\tau }f_i$ because each $f_i$ has rational Fourier coefficients. On the other hand, since Fourier coefficients of $f$ are in $\mathbb{Q}(\mu _{p-1})$, we have $f^{\tau }=f=\sum _{i=1}^Nc_if_i$. It follows from uniqueness of description of $f$ that $c_i^\tau =c_i$. The assertion follows.      
\end{proof}
We are now in a position to prove our main theorem. 
\begin{proof}[Proof of Theorem \ref{ThmM}] For any $F\in M_{k}(\Gamma _0^{(2)}(p),\chi )_{\mathbb{Q}(\mu _{p-1})}$, take a sequence of modular forms $\{G_{k_m}\in M_{k_m}(\Gamma _0^{(2)}(p),\chi ^{-1})\}$ constructed in Theorem \ref{Thm2}. We consider $FG_{k_m}\in M_{k+k_m}(\Gamma _0^{(2)}(p))_{\mathbb{Q}(\mu _{p-1})}$. Note here that each $k+k_m$ is even. Applying Lemma \ref{Lem3} to each $FG_{k_m}$, $FG_{k_m}$ is a $\mathbb{Q}(\mu _{p-1})$-linear combination of elements of $M_{k+k_m}(\Gamma _0^{(2)}(p))_{\mathbb{Q}}$. Hence, $(FG_{k_m})^\sigma =F^{\sigma }G_{k_m}^{\sigma }$ is a $p$-adic Siegel modular form according to Theorem \ref{Bo-Na}. Since $G_{k_m}^{\sigma }$ tends to $1$, $F^{\sigma }G_{k_m}^{\sigma }$ tends to $F^{\sigma }$. Thus $F^{\sigma }$ is a $p$-adic Siegel modular form. This completes the proof of Theorem \ref{ThmM}. 
\end{proof}

%%%%%%%%%%%%%%%%%%%%%%%%%%%%%%%%%%%%%%%%%%%%%%%%%%%%%%%%%%%%%%%%%%%%%%
\section{For generalization}
\label{Gen}
In this section, we mention some remarks for generalization. 

If the following problem is affirmative, then we can generalize Theorem \ref{ThmM} to the case of any degree.   
\begin{Prob}
\label{Prob1}
Let $k$ be a positive integer and $p$ an odd prime. For any Dirichlet character $\chi $ modulo $p$ with $\chi (-1)=(-1)^k$, we take $\alpha \in \mathbb{Z}/(p-1)\mathbb{Z}$ such that $\chi ^{\sigma }=\omega ^{\alpha }$. Then, does there exist a sequence of Siegel modular forms $\{G_{k_m}\in M_{k_m}(\Gamma _0^{(n)}(p),\chi )_{\mathbb{Q}(\mu _{p-1})}\}$ such that 
\begin{align*}
\lim _{m\to \infty }G_{k_m}^{\sigma }=1\quad (p{\text -adically})?
\end{align*}
\end{Prob}
Now we raise one more question which is equivalent to this problem.
\begin{Prob}
\label{Prob2}
Let $p$, $\chi $ and $\alpha $ be same as above. We take an integer $a$ such that $a\equiv -\alpha$ mod $p-1$. Then, does there exist a modular form $G_{a}\in M_{a}(\Gamma _0^{(n)}(p),\chi )_{\mathbb{Q}(\mu _{p-1})}$ such that 
\begin{align*}
G_{a}^{\sigma }\equiv 1 \bmod{p}? 
\end{align*}
\end{Prob}
\begin{Rem}
(1) If $\alpha =0$ (i.e. $p-1|a$), then this problem is affirmative by B\"ocherer-Ngaoka's result. \\
(2) If this problem is affirmative, then we can solve Probem \ref{Prob1} affirmatively by putting $G_{k_m}:=G_a^{p^m}$. 
\end{Rem}

\section*{Acknowledgment}
The author would like to thank Professor S.~Nagaoka for suggesting this problem. Macro 1
The fact of Lemma \ref{Lem3} was communicated by Professor S.~B\"ocherer during his stay at Kinki university. The author would like to thank Professor S.~B\"ocherer. 

%%%%%%%%%%%%%%%%%%%%%%%%%%%%%%%%%%%%%%%%%%%%%%%%%%%%%%%%%%%%%%%%%%%%%%%%%%%%%%%%%%%%%%%%%%%%%%%%%%%%%%%%%%%%%%%%%%%%%%%%%%%%%%%%%%%%%%%%%%%%%%%%%%%%%%%%%%%%%%%%



\end{document}